\documentclass[12pt,a4paper]{amsart}

\usepackage[utf8]{inputenc}
\usepackage[english]{babel}

\usepackage{latexsym}
\usepackage{amsfonts}
\usepackage{amsmath}
\usepackage{amssymb}
\usepackage{graphicx}
\usepackage{color}
\usepackage{enumitem}

\usepackage{amsthm}
\usepackage{pstricks}
\usepackage{csquotes}
\usepackage{fullpage}
\usepackage{url}
\usepackage[ocgcolorlinks, linkcolor=red]{hyperref}

\usepackage{algorithm}
\usepackage{algcompatible}

\newcommand{\R}{\mathbb{R}}
\newcommand{\N}{\mathbb{N}}

\newcommand{\C}{\mathbb{C}}

\newcommand{\ov}[1]{\overline{#1}}

\newcommand\supp{\operatorname{supp}} 
 
\newcommand\spec{\operatorname{Spec}}

\newcommand\p{\partial}

\newcommand\osupp{\operatorname{osupp}}

\newtheorem{thm}{Theorem}[section]

\newtheorem{lem}[thm]{Lemma}

\theoremstyle{definition}

\numberwithin{equation}{section}

\title{Shape reconstruction of inclusions based on noisy data via monotonicity methods for the time harmonic elastic wave equation}

\author{Sarah Eberle-Blick}

\thanks{E-Mail: eberle@math.uni-frankfurt.de, Institute of Mathematics, Goethe-University Frankfurt, Frankfurt am Main, Germany}

\date{}

\begin{document}

\maketitle

\begin{abstract}
In this paper, we extend our research concerning the standard and linearized monotonicity methods for the inverse problem of the time harmonic elastic wave equation and introduce the modification of these methods for noisy data. In more detail, the methods must provide consistent results when using noisy data in order to be able to perform simulations with real world data, e.g., laboratory data. We therefore consider the disturbed Neumann-to-Dirichlet operator and modify the bound of the eigenvalues in the monotonicity tests for reconstructing unknown inclusions with noisy data. In doing so, we show that there exists a noise level $\delta_0$ so that the inclusions are detected and their shape is reconstructed for all noise levels $\delta<\delta_0$. Finally, we present some numerical simulations based on noisy data. 
\end{abstract}

%\tableofcontents

\section{Introduction} 
\noindent
Inverse elasticity problems are used in a wide variety of areas and include medical (see e.g. \cite{BG04}) and geophysical (described in \cite{WACSMCCFSZ03}) applications, but especially the reconstruction of inclusions in material analysis (as shown in \cite{ABGKLW15} and the references therein). 
In this paper we focus on the shape reconstruction based on monotonicity methods for the time harmonic wave equation developed in \cite{EP24} and \cite{EP24a} and extend these methods to noisy data which is of great importance if we deal with data from real measurements, since there are always measurement errors.
\\
\\
Similar as in our previous papers \cite{EP24} and \cite{EP24a}, we introduce the direct problem of the time harmonic problem as follows: 
Let $\Omega \subset \R^3$ be an isotropic elastic body, where the properties in its linear regime are described via the Lamé parameters $\lambda$ and $\mu$. By applying time harmonic oscillations on $\Omega$, we are led to the Navier equation. In more detail, this equation gives the displacement field  $u : \Omega \to \R^3$, $u \in H^1(\Omega)^3$ of the solid body $\Omega$, due to disturbances. We specifically deal with a material body $\Omega \subset \R^3$ with a $C^{1,1}$-boundary and introduce the 
Navier equation by means of the following boundary value problem
\begin{align}  \label{eq_bvp1}
\begin{cases}
\nabla \cdot (\C\,  \hat \nabla u )  + \omega^2\rho u &=\,\, 0, \quad\text{in}\,\,\Omega,\\
\hspace{1.8cm}(\C\,  \hat \nabla u ) \nu  &=\,\, g, \quad\text{on}\,\,\Gamma_N, \\
\hphantom{\hspace{1.8cm}(\C\,  \hat \nabla u ) } u  &=\,\, 0, \quad\text{on}\,\,\Gamma_D, %\\
% \text{boundary condition},
\end{cases}
\end{align}
where $\Gamma_N$ and $\Gamma_D$ are such that
$$
\Gamma_N, \Gamma_D \subset \p\Omega \text{ are open }, \qquad \Gamma_N \neq \emptyset, \qquad
\p \Omega = \ov{\Gamma}_N \cup \ov{\Gamma}_D,
$$
and where $\hat \nabla u  = \frac{1}{2}(\nabla u + (\nabla u)^T)$ is the symmetrization of the Jacobian
or the strain tensor, 
and $\C$ is the 4th order tensor defined by
\begin{equation} \label{eq_Cdef}
\begin{aligned}
(\C A)_{ij} = 2\mu A_{ij} + \lambda \operatorname{tr}(A) \delta_{ij},
\quad \text{ where } A \in \R^{3 \times 3},
\end{aligned}
\end{equation}
and $\delta_{ij}$ is the Kronecker delta. The Lamé parameters are specified via the scalar functions 
$\lambda,\mu \in L^\infty_+(\Omega)$, which determine
the elastic  properties of the material, $\rho \in L^\infty_+(\Omega)$ is the density of the material,
and $\omega \neq 0$ the angular frequency of the oscillation, and $\nu$
is the outward pointing unit normal vector to the boundary $\p \Omega$. 
The vector field $g \in L^2(\Gamma_N)^3$ can be understood as the source
of the oscillation, and since $\C \,\hat \nabla u$ equals by Hooke's law to the Cauchy stress tensor, we see 
that the boundary condition $g$ specifies the traction on the surface $\p \Omega$.
\\
\\
We assume that $\omega \in \R$ is not a resonance frequency, which is a common assumption, since then the problem \eqref{eq_bvp1} has a unique solution for a given boundary condition $g\in L^2(\Gamma_N)^3$. Furthermore, we define the Neumann-to-Dirichlet map
$\Lambda:L^2(\Gamma_N)^3 \to L^2(\Gamma_N)^3$, as
\begin{equation} \label{eq_ND_map}
\begin{aligned}
\Lambda: g \mapsto u|_{\Gamma_N}. 
\end{aligned}
\end{equation}
Thus, $\Lambda$ maps the traction to the displacement $u|_{\Gamma_N}$ on the boundary.
\\
\\
Our aim is to formulate the standard and linearized monotonicity tests for noisy data. This problem can be handled by modifying the bound of the eigenvalues in the monotonicity tests to obtain the same reconstruction as without noise, as long as the noise does not exceeds the maximal noise level $\delta_0$. Please note, that the maximal noise level can be increased by using more boundary loads as shown in \cite{EH23} for the static case. As can be seen in Figure \ref{standard_noise_0} and \ref{linearized_noise_0}, noiseless reconstruction results in the correct detection of the inclusions. Both monotonicity methods can handle small amounts of noise with the standard monotonicity method being more robust regarding noise than the linearized method. Contrary, the standard method is much slower. Those results are similar to the stationary case (see \cite{EH21}).
\\
\\
\noindent
Before, we extend the monotonicity methods from \cite{EP24} and \cite{EP24a} for noisy data, we give a short overview of the methods applied so far:
\\
\\
Most of the methods are iterative (see e.g. \cite{SFHC14}, \cite{B18}, \cite{BYZ19}) and based on regularization techniques for minimization problems but there are also non-iterative methods such as linear sampling methods and factorization methods. In more detail, \cite{LWWZ16} develop a continuation method with respect to the frequency for the inverse problem in 2D to reconstruct inclusions. The reconstruction of constitutive parameters in isotropic linear elasticity from noisy full-field measurements (measurements inside the whole domain under consideration) using a regularization technique can be found in \cite{BBIM04}. In \cite{AB19}, an error in constitutive equation approach (ECE) to derive the Lamé parameters for time harmonic elasticity is examined. The method is based on the minimization of an energy-based cost functional. The authors of \cite{SFHC14} recover piecewise constant Lamé parameters for an unknown fixed density model. They provide a Lipschitz stability result for the inverse problem as well as a multi-level inversion scheme. A reconstruction of the boundary of a scatterer in a homogeneous medium from far-field data is considered in \cite{JL19} and \cite{LLS19}. To this end, they develop a sampling method in 2D, whereas \cite{HMSY20} propose a factorization method in 2D and 3D. In \cite{HKS12}, \cite{HLZ13} and \cite{EH19} a factorization method is described as well. Here, the focus lies on the inverse elastic scattering problem from periodic rigid structures in 2D. In \cite{CKAGK06} inverse elastics scattering problems are considered by means of the factorization method, too. 
In addition, the monotonicity method was adopted to the elasto-oscillatory inverse problem by the author in \cite{EP24} and \cite{EP24a}.
\\
\\
For the three-dimensional case, uniqueness theorems for the direct and inverse obstacle scattering problem are introduced in \cite{HH93}. A uniqueness theorem in inverse scattering of elastic waves based on the far-field is shown in \cite{H93} for piecewise constant density and constant Lamé parameters. A further uniqueness result can be found in \cite{JMRY03}, where only $\mu$ is considered, i.e., $\lambda=0$. In addition, the background shear modulus $\mu_0$ and $\rho=\rho_0$ are assumed to be constant. A Lipschitz type stability estimate assuming that the density is piecewise constant is shown in \cite{BHQ13} for constant Lamé parameters and in \cite{BHFVZ17} for piecewise constant Lamé parameters. 
Further on, we want to mention, the work \cite{EP24}, where the localized potentials were analyzed and their existence was proven by the Runge approximation.
\\
\\
The basic idea of the monotonicity method has been first worked out and numerically tested in \cite{TR02} and \cite{Ta06}. In addition, the monotonicity methods have already been successfully implemented in theory and practice for electrical impedance tomography (see, e.g., \cite{HU13}, \cite{HM16}), the Helmholtz equation (see \cite{HPS19a} and \cite{HPS19b}) as well as stationary (see \cite{EH21}, \cite{EH22}, \cite{EH23} and \cite{EHMR21} for Lipschitz stability) and time harmonic elasticity (see \cite{EP24} and \cite{EP24a}). 
\\
\\
In this paper we focus on the examination of the monotonicity methods for noisy data and the corresponding numerical simualtions, which are required if we deal, e.g., with real world data of a lab experiment (see \cite{EM21} for the stationary case).
\\
\\
The paper is organized as follows: We start with introduction of the inverse problem as well as an overview of the methods applied so far for solving this problem. In Section \ref{sec_defs}, we state the required definitions and summarize the results concerning the standard and linearized monotonicity methods in Section \ref{summary}. After that, we come to the main part of this paper, where we extend the monotonicity methods to noisy data and prove the resulting monotonicty tests. Finally, we present several numerical examples based on noisy data.

\section{Definitions and  other preliminaries} \label{sec_defs}

\noindent
In order to provide the corresponding background, we summarize the requiured definitions and preliminaries. We want to remark that the notations are to a large extent the same notations as in \cite{EP24}.
The following definitions related to function spaces are used through out the paper. 
The space $H^k(\Omega)$ denotes the $L^2(\Omega)$ based Sobolev space with $k$ weak derivatives.
In addition, we define
$$	
L^\infty_+(\Omega) := \big\{ f \in L^\infty(\Omega) \;:\; \operatorname{essinf}_\Omega f > 0 \big\}.
$$
We use the notation $Z^n$  for a function space $Z$
with $Z^n := Z \times \dots \times Z$, where the right hand side  contains $n$ copies of $Z$.
We define the  $L^2$-inner product by $(\cdot,\cdot)_{L^2}$, so that 
$$
(u,v)_{L^2(\Omega)^{n\times m}} := \int_\Omega u : v\,dx, \quad u,v \in L^2(\Omega)^{n \times m},\quad n,m\geq 1,
$$
where $:$ is the Frobenius inner product defined below.
For orthogonality with respect to the inner product $(\cdot,\cdot)_{L^2}$, we apply the notation $\perp$, unless otherwise stated,
so that 
$$
u \perp v \qquad \Leftrightarrow \qquad (u,v)_{L^2(\Omega)^n} = 0, \quad  \text{ when } u,v \in L^2(\Omega)^n.
$$
\noindent
For the well-posedness of problem \eqref{eq_bvp1} the bilinear form $B$ related to equation \eqref{eq_bvp1} is considered, which is given by
\begin{align}  \label{eq_weak}
B(u,v)  := -\int_\Omega 
2 \mu \hat \nabla u :\hat \nabla v + \lambda \nabla \cdot u \nabla \cdot v - \omega^2\rho u\cdot v\,dx, 
\end{align}
for all $u,v \in H^1(\Omega)^3$. The  Frobenius inner product $A:B$ is defined as
$$
A:B = \sum A_{ij}B_{ij},  \qquad A,B \in \R^{m\times n}.
$$
Note that the Euclidean norm on $\R^{m \times n}$, $m,n \in \N$, is given by
$|A| = (A:A)^{1/2}$, for $A \in \R^{m \times n}$.
We further introduce the notation
$$
L_{\lambda,\mu,\rho} u := \nabla \cdot ( \C \hat \nabla u)+ \omega^2\rho u. 
$$
We want to remark that in an isotropic medium characterized by the Lamé parameters the above 
equation simplifies to
$$
L_{\lambda,\mu,\rho} u = \nabla \cdot ( 2 \mu \hat \nabla u + \lambda (\nabla \cdot u) I ) + \omega^2\rho u.
$$
A weak solution to \eqref{eq_bvp1} is defined as a $u \in H^1(\Omega)^3$,
for which $u|_{\Gamma_D} = 0$ and 
\begin{align}  \label{eq_weak2}
B(u,v)  = - \int_{\Gamma_N} g\cdot v \,dS,
\qquad \forall v \in \{ u \in H^1(\Omega)^3 \,:\, u|_{\Gamma_D} = 0 \}.
\end{align}
\noindent
For the existence and uniqueness of a weak solution to \eqref{eq_bvp1}, when $\omega$
is not a resonance frequency see Corollary 3.4 in \cite{EP24}.
\\
\\
The existence and uniqueness of a weak solution to \eqref{eq_bvp1} implies that 
the Neumann-to-Dirichlet map $\Lambda$ given by \eqref{eq_ND_map} is well defined. $\Lambda$ is related to $B$ as  follows.
When the material parameters are regular, and $u$ solves \eqref{eq_bvp1} with $g$, and
$v$ solves \eqref{eq_bvp1} with $h$, we see by integrating by parts that
\begin{align}  \label{eq_NDmap}
B(u,v)  = -\int_{ \Gamma_N} g \cdot v \,dS = -( g, \Lambda h)_{L^2(\Gamma_N)^3}. 
\end{align}
We abbreviate the  boundary condition in \eqref{eq_bvp1} by
\begin{align*} 
\gamma_{\mathbb{C},\Gamma} u  = (\C \, \hat \nabla u ) \nu |_{\Gamma},
\end{align*}
or with $\gamma_\mathbb{\C} u $ if the boundary is clear from the context. 
Note that these notations are formal when $u \in H^1(\Omega)^3$, since we cannot in general take the trace of an 
$L^2(\Omega)^3$ function. In the low regularity case we understand the boundary condition in a weak sense.
We define $\gamma_{\mathbb{C}} u \in L^2(\Gamma_N)^3$, 
with $u \in H^1(\Omega)^3$ that solves \eqref{eq_bvp1}, as the element in $L^2(\Gamma_N)^3$, for which
\begin{align*} 
-(\gamma_{\mathbb{C}} u, \, \varphi|_{\Gamma_N}  )_{L^2(\Gamma_N)^3}
= B(u,\varphi), \qquad \forall \varphi \in H^1(\Omega)^3. 
\end{align*}

%For the convenience of the reader we also list the following integration by parts formulas
%that are used through out the paper. Firstly we have the matrix form of the divergence theorem
%$$ 
%\int_\Omega \nabla \cdot  A  u \,dx = - \int_\Omega A:\hat \nabla u \,dx + \int_{\p \Omega} A \nu \cdot u\,dS,
%\quad A \in H^1(\Omega)^{n\times n},
%$$
%where $u,v \in H^1(\Omega)^3$. See \cite{Ci88} p. xxix. We also use the integration by parts  formula
%\begin{align*} %\label{eq_intbyparts1}
%\int_\Omega \nabla \times u \cdot v \,dx = \int_\Omega  u \cdot \nabla \times v \,dx
%+
%\int_{\p \Omega} (\nu \times u) \cdot v\,dS, \quad u,v \in H^1(\Omega)^3.
%\end{align*}
\noindent
\\
Next we give a definition of the notion of the outer support of a function or a set.
The outer support (with respect to $\p\Omega$) of a measurable function $f: \Omega \to \R$   is defined as 
$$
\operatorname{osupp} (f) := \Omega \setminus \bigcup \, \big \{  U \subset \Omega \,:\, U  
\text{ is relatively open and connected to $\p \Omega$},
f|_U \equiv 0  \big\}.
$$
For more on this see \cite{HU13}.
It will be convenient to extend this definition to sets. We define the outer support of a measurable
set $D \subset \Omega$ 
(with respect to $\p\Omega$) as
\begin{align}  \label{eq_def_osupp}
\osupp (D) := \osupp(\chi_D),
\end{align}
where $\chi_D$ is the characteristic function of the set $D$.

%%%%%%%%%%%%%%%%%%%%%%%%%%%%%%%%%%%%%%%%%%%%%%%%%%%%%%%%%%%%%%%%%%%%%%%%%%%%%%%%%%%%%%%%%%%%%%%%%%%%%%%%%%%%%%%

\section{Summary of the monotonicty methods}\label{summary}
\noindent
We give a short summary of the two monotonicity methods we deal with. We will take a look at the standard (see Subsection \ref{standard}) as well as the linearized (see Subsection \ref{sec_mono_shape}) monotonicity method with which we can reconstruct the set $\osupp(D)$, where 
\begin{align*}
D = \supp(\lambda-\lambda_0) \cup \supp(\mu-\mu_0) \cup \supp(\rho- \rho_0),
\end{align*}
as we did in \cite{EP24} and \cite{EP24a}.
\\
\\
We introduce the mixed eigenvalue problem (for comparison see \cite{EP24})
 \begin{align}  \label{eq_Neumann}
 \begin{cases}
 \nabla \cdot (\C \, \hat \nabla \varphi_k )  + \omega^2\rho\varphi_k &= \quad \sigma_k \varphi_k,\\
 \quad\quad\quad\quad\quad\quad \gamma_\mathbb{C} \varphi_k  |_{\Gamma_N} &= \quad 0, \\
 \quad\quad\quad\quad\quad\quad\quad \varphi_k  |_{\Gamma_D} &=\quad 0
 \end{cases}
 \end{align}
 \noindent
 and use the quantity  $d(\lambda,\mu,\rho)$, which we define  as
 \begin{align}  \label{eq_def_d}
 d(\lambda,\mu,\rho) := \text{ the number of $\sigma_k > 0$ in problem \eqref{eq_Neumann} counted with multiplicity.}
 \end{align}
 \noindent
 \\
We will consider inhomogeneities in the material parameters of the following type.
Let $D_1, D_2, D_3 \Subset \Omega$. 
We will now assume that $\lambda,\mu, \rho \in L_+^\infty(\Omega)$ 
are such that
\begin{equation} \label{eq_lambdaMuRho}
\begin{aligned} 
\lambda(x) &= \lambda_0 + \chi_{D_1}(x) \psi_\lambda (x), \qquad \psi_\lambda \in L^\infty(\Omega),
\quad \psi_\lambda(x) > n_1, \\
\mu(x) &= \mu_0 + \chi_{D_2}(x) \psi_\mu(x), \qquad \psi_\mu \in L^\infty(\Omega),\quad \psi_\mu(x) > n_2, \\
\rho(x) &= \rho_0 - \chi_{D_3}(x) \psi_\rho(x), \qquad \psi_\rho \in L^\infty(\Omega),\quad n_3 < \psi_\rho(x) < N_3, 
\end{aligned}
\end{equation}
where the constants $\lambda_0,\mu_0,\rho_0 > 0$ and the bounds  $n_1,n_2, n_3 > 0$ and $\rho_0 > N_3$.
The coefficients $\lambda,\mu$ and $\rho$  model inhomogeneities in an otherwise homogeneous background medium
given by the coefficients $\lambda_0,\mu_0$ and $\rho_0$. 
\\
\\
Next we define the test coefficients $\lambda^\flat,\mu^\flat$ and $\rho^\flat$. Let $B \subset \Omega$ be an open set. We let
\begin{equation} \label{eq_testCoeff}
\begin{aligned} 
\lambda^\flat(x) &= \lambda_0 + \alpha_1 \chi_{B}(x), \\
\mu^\flat(x) &= \mu_0 + \alpha_2\chi_{B}(x), \\
\rho^\flat(x) &= \rho_0 - \alpha_3\chi_{B}(x), \\
\end{aligned}
\end{equation}
where $\alpha_j >  0$ are constants and $\chi_{B}$ is the characteristic function w.r.t. the test inclusion $B$.

\subsection{Standard monotonicty test}\label{standard}
The following theorem gives the background for the standard monotonicity tests for exact data.
\begin{thm}[see \cite{EP24}]
Let $D := D_1 \cup D_2 \cup D_3$, where the sets are as in \eqref{eq_lambdaMuRho} and
$B \subset \Omega$  and $\alpha_j > 0$ be as in \eqref{eq_testCoeff}, 
and set $\alpha:=(\alpha_1,\alpha_2,\alpha_3)$.  Let $M_s\in\mathbb{R}$ be defined as 
 \begin{align*}
 M_s:=d(\lambda_0,\mu_0,\rho_0),
 \end{align*}
 \noindent
 where $d(\lambda_0,\mu_0,\rho_0)$ is the number of positive eigenvalues of $L_{\lambda_0,\mu_0,\rho_0}$ as defined in (\ref{eq_def_d}).
The following holds: \\
\begin{enumerate}
\item Assume that  $B \subset D_j$, for $j \in J$, for some $J \subset\{1,2,3\}$.  
Then for all $\alpha_j$ with $\alpha_j \leq n_j$, $j \in J$, and $\alpha_j = 0$, 
$j \notin J$, the map $\Lambda^\flat - \Lambda$ has at most $M_s$ negative eigenvalues. \\

\item If $B \not \subset \osupp(D)$, then for all $\alpha$, $|\alpha| \neq 0$, 
the map  $\Lambda^\flat - \Lambda$ has more than $M_s$ negative eigenvalues,

\end{enumerate}
where $\Lambda$ is the Neumann-to-Dirichlet map for the coefficients in \eqref{eq_lambdaMuRho}
and $\Lambda^\flat$ is the Neumann-to-Dirichlet map for the coefficients in \eqref{eq_testCoeff},
and where the eigenvalues of $\Lambda^\flat-\Lambda$ are counted with multiplicity.
\end{thm}

\subsection{Linearized monotonicity test} \label{sec_mono_shape}
Next, we present the main theorem underlying our linearized monotonicity tests, where we introduce the associated bilinear form of the Fréchet derivative $\Lambda^\prime_{\lambda,\mu,\rho}$ as in \cite{EP24a}:
\begin{align*}
\left(\Lambda^\prime_{\lambda,\mu,\rho}[\hat{\lambda},\hat{\mu},\hat{\rho},] g,f \right)_{L^2(\Gamma_{\textup{N}})^3}
= -\int_{\Omega}2\hat{\mu} \hat{\nabla}u_g :\hat{\nabla}u_f +\hat{\lambda}
\nabla\cdot u_g \nabla\cdot u_f - \omega^2 \hat{\rho} u_g \cdot u_f\,dx.
\end{align*}

\begin{thm}[see \cite{EP24a}]\label{thm_Lin_inclusionDetection}
Let $D := D_1 \cup D_2 \cup D_3$, where the sets are as in \eqref{eq_lambdaMuRho} and
$B \subset \Omega$  and $\alpha_j > 0$, 
and set $\alpha:=(\alpha_1,\alpha_2,\alpha_3)$. Let
$$
\mathcal{M}:= \# \big\{\sigma \in \spec(\Lambda_0 + \Lambda'_0[\alpha_1,\alpha_2,-\alpha_3] - \Lambda)
\;:\; \sigma < 0 \big\},
$$
where $\Lambda_0$ and $\Lambda$ are the NtD-maps for the coefficients $\lambda_0,\mu_0,\rho_0$ and $\lambda,\mu,\rho$ respectively,
and where $\Lambda'_0[\alpha_1,\alpha_2,-\alpha_3] := \Lambda'_0[\alpha_1\chi_B,\alpha_2\chi_B,-\alpha_3\chi_B]$. 
There exists a $\gamma_0 > 0$ such that the following holds: \\
\begin{enumerate}

\item 
Assume that  $B \subset D_j$, for $j \in J$, for some $J \subset\{1,2,3\}$.  
Then for all $\alpha_j$ with $\alpha_j \leq \gamma_0$, $j \in J$, and $\alpha_j = 0$, 
$j \notin J$, we have that  $\mathcal{M}< \infty$.
\\

\item
If $B \not \subset \osupp(D)$, then for all $\alpha$, $|\alpha| \neq 0$, $\mathcal{M} = \infty$.\\

\end{enumerate}
\end{thm}

\begin{lem}\label{lem_Lin_inclusionDetection}
Let the assumptions of Theorem \ref{thm_Lin_inclusionDetection} hold. Let
$$
\mathcal{M}_k:= \# \big\{\sigma \in \spec(\Lambda_0 + \Lambda'_0[\alpha_1\chi_{B_k},\alpha_2\chi_{B_k},-\alpha_3\chi_{B_k}] - \Lambda)
\;:\; \sigma < 0 \big\},
$$
where $\Lambda_0$ and $\Lambda$ are the NtD-maps for the coefficients $\lambda_0,\mu_0,\rho_0$ and $\lambda,\mu,\rho$ respectively. Let further $\mathcal{B}=\left\{B_k\ |\ B_k\subset \Omega\right\}$ be a fixed finite set. Then there exists a $\gamma_0 > 0$ and a $0\leq M_l$ such that the following holds: \\

 \begin{enumerate}

\item 
Assume that  $B_k \subset D_j$, for $j \in J$, for some $J \subset\{1,2,3\}$.  
Then for all $\alpha_j$ with $\alpha_j \leq \gamma_0$, $j \in J$, and $\alpha_j = 0$, 
$j \notin J$, we have that  $\mathcal{M}_k\leq M_l$.
\\

\item
If $B_k \not \subset \osupp(D)$, then for all $\alpha$, $|\alpha| \neq 0$, $\mathcal{M}_k > M_l$.

% \item 
% \textcolor{blue}{Let $\rho = \rho_0$, and that $\Gamma_N = \p \Omega$.}
% If $B \subset \Omega \setminus (\osupp(\mu) \cup \osupp(\rho))$, and we choose
% $\alpha_2 \neq 0$ or $\alpha_3 \neq 0$, then  $\mathcal{N} = \infty$.
% \\
%\item 
%Assume $\rho = \rho_0$, and that $\Gamma_N = \p \Omega$.
%If $B \not \subset \osupp(\mu)$, and we choose
%$\alpha = (0,\alpha_2,0) \neq 0$, then  $M_k \geq M_0$.
\end{enumerate}
\end{lem}

\begin{proof}
Assume that $B_k \subset D_j$, for $j \in J$, for some $J \subset\{1,2,3\}$.  
Then for all $\alpha_j$ with $\alpha_j \leq \gamma_0$, $j \in J$, and $\alpha_j = 0$, 
$j \notin J$, we have that  $\mathcal{M}_k$ is finite  due to Theorem \ref{thm_Lin_inclusionDetection}. Hence, we set 
$$
M_l=\max_{\left\{k\ |\ \mathcal{M}_k<\infty\right\}}\mathcal{M}_k
$$
$M_l<\infty$, since the set $\mathcal{B}$ is finite as well.
Further $\mathcal{M}_k\leq M_l$ due to the construction of $M_l$.\\
\\
On the other hand, if $B_k \not \subset \osupp(D)$, then for all $\alpha$, $|\alpha| \neq 0$, $\infty=\mathcal{M}_k > M_l$ due to Theorem \ref{thm_Lin_inclusionDetection}.
\end{proof}

 \section{Monotonicity tests for noisy data}
 \noindent
 \\
 Next, we formulate the monotonicity tests for noisy data and prove the corresponding theorems.

 \subsection{Standard monotonicity test}
We start with the standard monotonicity tests for noisy data.
 \begin{thm}\label{theo_standard_noisy}
 Let $D:=D_1\cup D_2 \cup D_3$, where the sets are as in (\ref{eq_lambdaMuRho}), $B\subset \Omega$, $\alpha_j>0$ be as in (\ref{eq_testCoeff}), and set $\alpha:=(\alpha_1,\alpha_2,\alpha_3)$.
 Let $M_s\in\mathbb{R}$ be defined as 
 \begin{align*}
 M_s:=d(\lambda_0,\mu_0,\rho_0),
 \end{align*}
 \noindent
 where $d(\lambda_0,\mu_0,\rho_0)$ is the number of positive eigenvalues of $L_{\lambda_0,\mu_0,\rho_0}$ as defined in (\ref{eq_def_d}). Further, let 
 \begin{align*}
 \Vert \Lambda^\delta(\lambda,\mu,\rho) - \Lambda(\lambda,\mu,\rho)\Vert < \delta.
 \end{align*}
 \noindent
 Then there exists a maximal noise level $\delta_0>0$, such that for all $0< \delta <\delta_0$ we have the following statements:
 \begin{itemize}
 \item[(1)] Assume that $B\subset D_j$ for $j\in J$ for some $J\subset \lbrace 1,2,3\rbrace$. Then for all $\alpha_j$ with $\alpha_j\leq n_j$, $j\in J$ and $\alpha_j=0$, $j\not\in J$, the map
$\Lambda^\flat - \Lambda^\delta$
 has at most $M_s$ eigenvalues smaller than $-\delta$.
 \item[(2)] If $B\not\subset \osupp(D)$, then for all $\alpha,\,|\alpha|\neq 0$, the map
$\Lambda^\flat - \Lambda^\delta$
 has more than $M_s$ eigenvalues smaller than $-\delta$,
 \end{itemize}
 \noindent
 where $\Lambda^\delta$ is the Neumann-to-Dirichlet map for the noisy problem.
 \end{thm}
 \noindent
 \begin{proof}
 We start with the case $B\not\subset D$:
\\
$\Lambda^\flat - \Lambda$ is compact and self-adjoint and by Theorem $6.5$ from \cite{EP24} there exists a $g\in V^\perp$ with
\begin{align*}
((\Lambda^\flat - \Lambda)g,g)\not\geq 0
\end{align*}
\noindent
with $\text{dim}(V)= M_s$. Hence, $\Lambda^\flat - \Lambda$ has more than $M_s$ negative eigenvalues. Let $\theta<0$ be the smallest eigenvalue with corresponding eigenvector $g\in V^\perp$. Then
 \begin{align*}
 &((\Lambda^\flat - \Lambda^\delta )g,g)\\
 &\leq ((\Lambda^\flat - \Lambda)g,g) + \delta\Vert g\Vert^2\\
 &=(\theta +  \delta)\Vert g\Vert^2\\
 &<-\delta \Vert g\Vert^2
 \end{align*}
 \noindent
 for all $0<\delta < \delta_0:=-\dfrac{\theta}{2}$.
 \\
 \\
On the other hand, if $B \subset D$, then for all $g\in V^\perp$
 \begin{align*}
 &((\Lambda^\flat - \Lambda^\delta )g,g)\\
 &\geq -\delta \Vert g\Vert ^2.
 \end{align*}
 \end{proof}
 \noindent
All in all, we end up with the following algorithm for the reconstruction based on noisy data with the standard monotonicity test:
\begin{algorithm}[H] 

\caption{Reconstruction of the inclusion $\text{osupp}(D) \subset \Omega$}\label{alg_stand}
\begin{algorithmic}[1]

\STATE  Choose a set $\mathcal{B} = \{ B \subset \Omega\}$ and set $\mathcal{A} = \{\}$.

\STATE  Choose $\tilde{M}_s$.

\FOR{  $B \in \mathcal{B}$}

\FOR{  $\Lambda^\flat$ with varied parameters by Theorem \ref{standard}}

\STATE Compute $\displaystyle\mathcal{M}_B := \sum_{\sigma_k < -\delta} 1$, where $\sigma_k$ are the eigenvalues of $\Lambda^\flat - \Lambda^\delta$ 

\IF {$ \mathcal{M}_B \leq \tilde{M}_s $}

\STATE 
Add $B$ to the approximating collection $\mathcal{A}$, since 
Theorem \ref{standard} suggests 
\STATE 
that $B \subset \osupp(D)$. 

\ELSE

\STATE Discard $B$, since $B \not \subset D_j$, $j=1,2,3$.

\ENDIF
\ENDFOR
\ENDFOR

\STATE  Compute the union of all elements in $\mathcal{A}$ and all components of 
$\Omega \setminus \cup \mathcal{A}$ not connected 
 to $\partial\Omega$. The resulting set is an approximation of $\text{osupp}(D)$.

\end{algorithmic}
\end{algorithm}

 \subsection{Linearized monotonicity test}
We continue with the linearized monotonicity tests.
 \begin{thm} \label{thm_Lin_inclusionDetection_noisy}
Let $D := D_1 \cup D_2 \cup D_3$, where the sets are as in \eqref{eq_lambdaMuRho} and let
$\mathcal{B}=\left\{B_k\ |\ B_k \subset \Omega\right\}$ be fixed and finite. Further, let $\alpha_j > 0$, 
and set $\alpha:=(\alpha_1,\alpha_2,\alpha_3)$. Let
$$
\mathcal{M}_k:= \# \big\{\sigma \in \spec(\Lambda_0 + \Lambda'_0[\alpha_1\chi_{B_k},\alpha_2\chi_{B_k},-\alpha_3\chi_{B_k}] - \Lambda)
\;:\; \sigma < -\delta \big\},
$$
where $\Lambda_0$ and $\Lambda$ are the NtD-maps for the coefficients $\lambda_0,\mu_0,\rho_0$ and $\lambda,\mu,\rho$ respectively. Further, let 
 \begin{align*}
 \Vert \Lambda^\delta(\lambda,\mu,\rho) - \Lambda(\lambda,\mu,\rho)\Vert < \delta.
 \end{align*}
 \noindent
 Then there exists a $\gamma_0 > 0$, a $0\leq M_l$ and a maximal noise level $\delta_0>0$, such that for all $0 < \delta <\delta_0$ we have the following statements: \\
\begin{enumerate}

\item 
Assume that  $B_k \subset D_j$, for $j \in J$, for some $J \subset\{1,2,3\}$.  
Then for all $\alpha_j$ with $\alpha_j \leq \gamma_0$, $j \in J$, and $\alpha_j = 0$, 
$j \notin J$, we have that  $\mathcal{M}_k \leq M_l$.
\\

\item
If $B_k \not \subset \osupp(D)$, then for all $\alpha$, $|\alpha| \neq 0$, $\mathcal{M}_k >M_l$.\\

% \item 
% \textcolor{blue}{Let $\rho = \rho_0$, and that $\Gamma_N = \p \Omega$.}
% If $B \subset \Omega \setminus (\osupp(\mu) \cup \osupp(\rho))$, and we choose
% $\alpha_2 \neq 0$ or $\alpha_3 \neq 0$, then  $\mathcal{N} = \infty$.
% \\
%\item 
%Assume $\rho = \rho_0$, and that $\Gamma_N = \p \Omega$.
%If $B \not \subset \osupp(\mu)$, and we choose
%$\alpha = (0,\alpha_2,0) \neq 0$, then  $M=\infty$.
%\\

\end{enumerate}
\end{thm}

\begin{proof}
 We start with the case $B_k\not\subset D$:
\\
$\Lambda_0 + \Lambda_0^\prime[\alpha_1\chi_{B_k},\alpha_2\chi_{B_k},-\alpha_3\chi_{B_k}]- \Lambda$ is compact and self-adjoint and by Theorem 6.1 from \cite{EP24a} as well as Lemma \ref{lem_Lin_inclusionDetection}, there exists a finite dimensional vector space $V $ of dimension $M_l$ and a $g\in V^\perp$ with
\begin{align*}
((\Lambda_0 + \Lambda_0^\prime[\alpha_1\chi_{B_k},\alpha_2\chi_{B_k},-\alpha_3\chi_{B_k}]- \Lambda)g,g)\not\geq 0.
\end{align*}
\noindent
 Hence, $\Lambda_0 + \Lambda_0^\prime[\alpha_1\chi_{B_k},\alpha_2\chi_{B_k},-\alpha_3\chi_{B_k}]- \Lambda$ has more than $M_l$ negative eigenvalues. Let $\theta<0$ be the smallest eigenvalue with corresponding eigenvector $g\in V^\perp$. Then
 \begin{align*}
 &((\Lambda_0 + \Lambda_0^\prime[\alpha_1\chi_{B_k},\alpha_2\chi_{B_k},-\alpha_3\chi_{B_k}] - \Lambda^\delta )g,g)\\
 &\leq ((\Lambda_0 + \Lambda_0^\prime[\alpha_1\chi_{B_k},\alpha_2\chi_{B_k},-\alpha_3\chi_{B_k}] - \Lambda)g,g) + \delta \Vert g\Vert^2\\
 &=(\theta + \delta)\Vert g\Vert^2\\
 &<-\delta \Vert g\Vert^2
 \end{align*}
 \noindent
 for all $0<\delta < \delta_0:=-\dfrac{\theta}{2}$, where $\theta<0$  is the smallest negative eigenvalue.
 \\
 \\
 On the other hand, if $B_k \subset D$, then via Lemma \ref{lem_Lin_inclusionDetection}, we have that for all $g\in V^\perp$
 \begin{align*}
 &((\Lambda_0 + \Lambda_0^\prime[\alpha_1\chi_{B_k},\alpha_2\chi_{B_k},-\alpha_3\chi_{B_k}] - \Lambda^\delta )g,g)\\
 &\geq -\delta \Vert g\Vert ^2.
 \end{align*}
%\\
%Finally, part (iii) can be shown in a similar way.
 \end{proof}
\noindent
\\
In order to close this section, we formulate the corresponding algorithm:
\begin{algorithm} 
\caption{Linearized reconstruction of  $\osupp(D) \subset \Omega$.}\label{alg_shapeInclusion}
\begin{algorithmic}[1]

\STATE  Choose a set $\mathcal{B} = \{ B \subset \Omega\}$ and set $\mathcal{A} = \{\}$.

\STATE  Choose $\tilde{M}_l$.

% \STATE  Compute $M_0$ using Theorem  \ref{thm_Lin_inclusionDetection}

\FOR{  $B \in \mathcal{B}$}

\FOR{  $\Lambda^\flat$ with parameters varied as suggested by Theorem \ref{thm_Lin_inclusionDetection}}

\STATE Compute $\displaystyle\mathcal{M}_B := \sum_{\sigma_k < -\delta} 1$, where $\sigma_k$ are the eigenvalues of 
\STATE $\Lambda_0 + \Lambda'_0[\alpha_1,\alpha_2,-\alpha_3] - \Lambda^\delta$ 

\IF {$ \mathcal{M}_B < \tilde{M}_l$}

\STATE 
Add $B$ to the approximating collection $\mathcal{A}$, since 
Theorem \ref{thm_Lin_inclusionDetection} suggests 
\STATE 
that $B \subset \osupp(D)$. 

\ELSE

\STATE Discard $B$, since by Theorem \ref{thm_Lin_inclusionDetection} $B \not \subset D_j$, $j=1,2,3$.

\ENDIF
\ENDFOR
\ENDFOR

\STATE  Compute the union of all elements in $\mathcal{A}$ and all components of 
$\Omega \setminus \cup \mathcal{A}$ not connected 
\STATE to $\p\Omega$. The resulting set is an approximation of $\osupp(D)$.

\end{algorithmic}
\end{algorithm}

\section{Numerical simulations} \label{sec_numerics}
\noindent
We introduce the noise matrix as
\begin{align*}
E=\dfrac{\tilde{E}}{\Vert \tilde{E} \Vert},
\end{align*}
\noindent
where $\tilde{E}$ is uniformly random on $[-1,1]$ and define the perturbed Neumann-to-Dirichlet operator via
\begin{align*}
\Lambda^\delta=\Lambda + \eta \Vert\Lambda\Vert E.
\end{align*}
\noindent
It holds that
\begin{align*}
\Vert \Lambda -\Lambda^\delta \Vert \leq \delta.
\end{align*}
\noindent
Indeed, we have
\begin{align*}
\Vert \Lambda^\delta - \Lambda \Vert
= \left \Vert \Lambda + \eta \Vert \Lambda\Vert E -\Lambda\right \Vert 
= \underbrace{\eta\Vert \Lambda\Vert}_{=\delta}.
\end{align*}
\noindent
This means for our monotonicity methods that in the standard tests
\begin{align*}
\Lambda^\flat - \Lambda^\delta
\end{align*}
\noindent
has at most $M_s$ eigenvalues smaller than $-\delta$ and in the linearized tests
\begin{align*}
\Lambda_0 + \Lambda_0^\prime[\alpha_1,\alpha_2,\alpha_3] - \Lambda^\delta
\end{align*}
\noindent
has at most $M_l$ smaller than $-\delta$ eigenvalues, if the test inclusion $B$ is part of the inclusion to be reconstructed.
\\
\\
Based on the two algorithms (see Algorithm \ref{alg_stand} and \ref{alg_shapeInclusion}), we present some numerical tests and examine an artificial test object with two inclusions (blue) shown in Figure \ref{testobject}. The size of our test object is $1\,m^3$.
\begin{figure}[H]
\centering 
\includegraphics[width=0.33\textwidth]{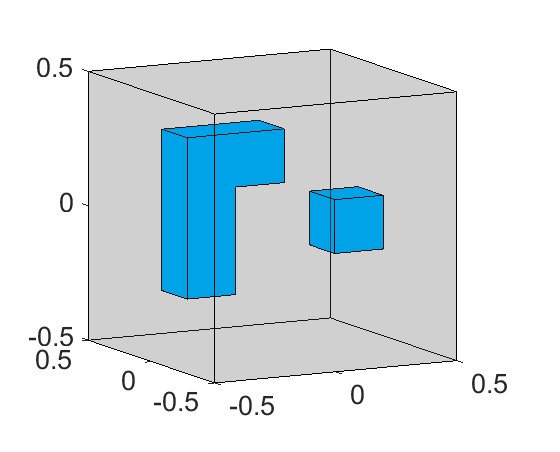}
\caption{Cube with two inclusions (blue) (similar as the model in \cite{EH21}).}\label{testobject}
\end{figure}
\noindent
The parameters of the corresponding materials are given in Table \ref{lame_parameter_mono}
\renewcommand{\arraystretch}{1.3} 
\begin{table} [H]
 \begin{center}
 \begin{tabular}{ |c|c| c |c |}  
\hline
 material & $\lambda_i$ & $\mu_i$ & $\rho_i$\\
  \hline
$i=0$: background &  $6\cdot 10^5$   &  $6\cdot 10^3$  & $3\cdot 10^3$ \\
 \hline
$i=1$: inclusion &  $2\cdot 10^6$ &  $2\cdot 10^4$  & $10^3$ \\
\hline
\end{tabular}
\end{center}
\caption{Lam\'e parameter $\lambda$ and $\mu$ in [$Pa$] and density $\rho$ in [$kg/m^3$].}
\label{lame_parameter_mono}
\end{table}
\noindent
Given an angular frequency $\omega$, the $s$-wavelength and $p$-wavelength for the homogeneous background material are defined via
\begin{align*}
l_p= 2\pi \dfrac{v_p}{k} \quad\text{and}\quad l_s= 2\pi \dfrac{v_s}{k}
\end{align*}
\noindent
with the velocities
$v_p=\sqrt{\frac{\lambda_0+2\mu_0}{\rho_0}}$ and  $v_s=\sqrt{\frac{\mu_0}{\rho_0}}$.

\subsection{Standard monotonicity test}
We start with the standard monotonicity tests and consider different noise levels $\eta$ for the frequency $\omega=50\,rad/s$ (see Figure \ref{standard_noise_0} - Figure \ref{standard_noise_0_025}).
In this example, we assume that the bottom of the test object is fixed and each of the remaining surfaces is divided into $10\times10$ patches, where we apply a normal force on each patch which results in $500$ boundary loads.

\begin{figure}[H]
\centering 
\includegraphics[width=0.38\textwidth]{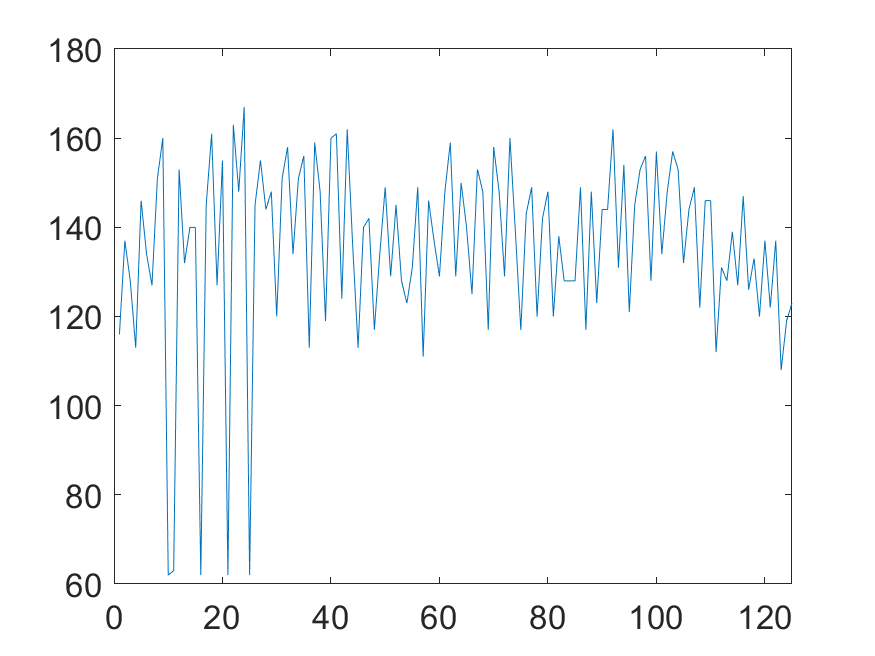}
\includegraphics[width=0.38\textwidth]{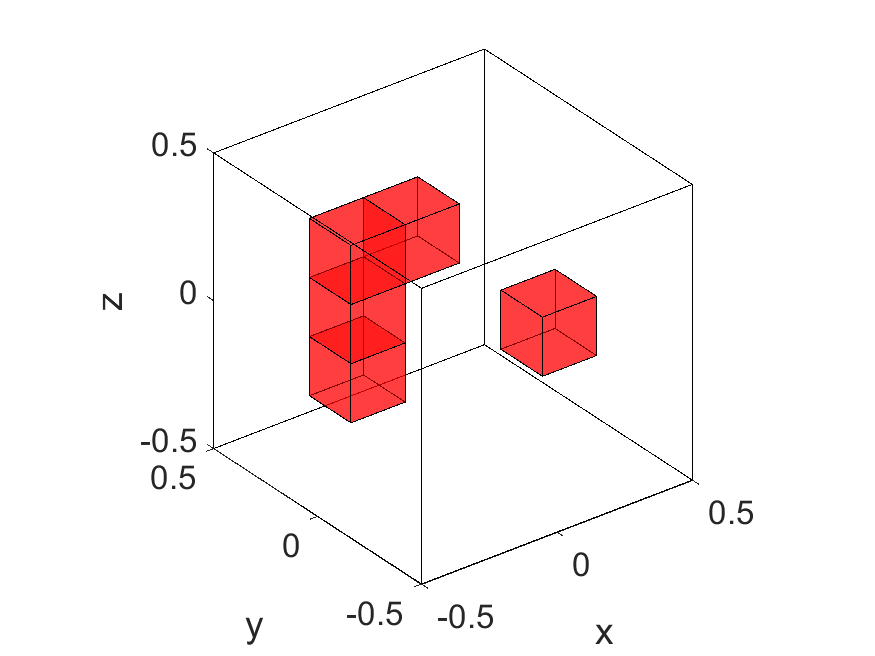}
\caption{Plot of the number of negative eigenvalues (left) and the reconstruction (right) for $\omega=50\,rad/s$ ($l_p=1.79\,m$, $l_s=0.18\,m$ for the homogenous background
material) for noise level $\eta=0$ and $\tilde{M}_s=107$.}\label{standard_noise_0}
\end{figure}

\begin{figure}[H]
\centering 
\includegraphics[width=0.38\textwidth]{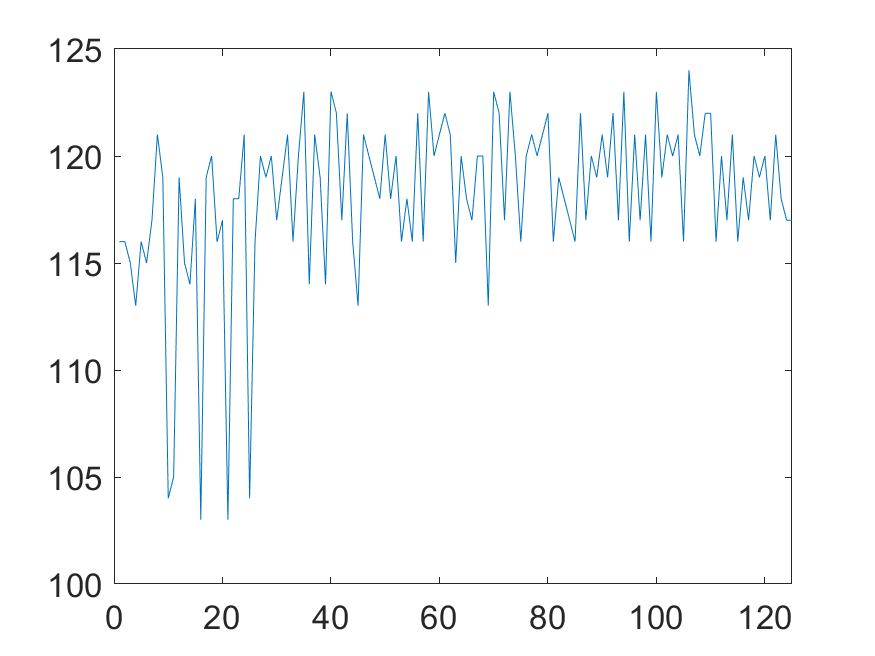}
\includegraphics[width=0.38\textwidth]{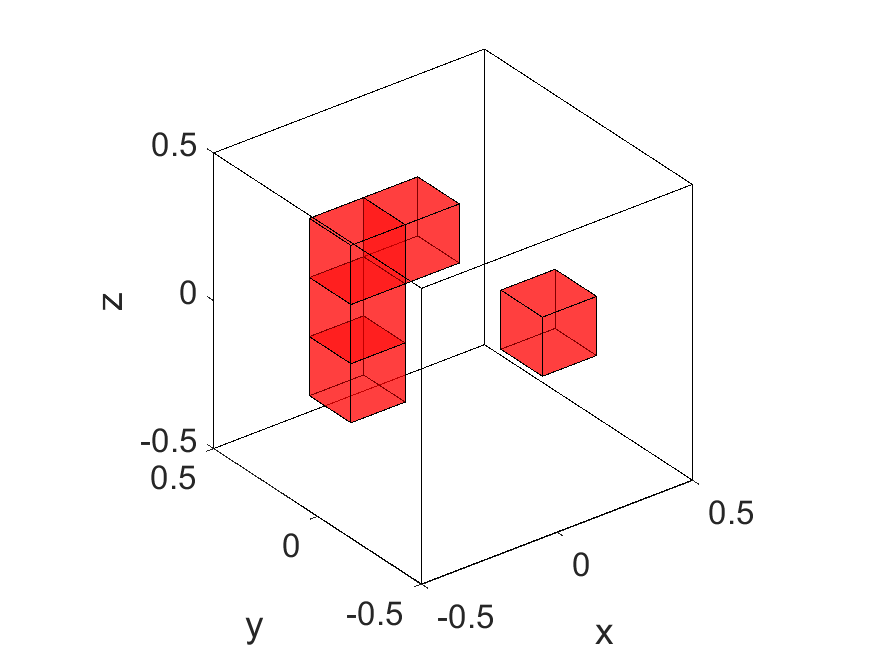}
\caption{Plot of the number of negative eigenvalues (left) and the reconstruction (right) for $\omega=50\,rad/s$ ($l_p=1.79\,m$, $l_s=0.18\,m$ for the homogenous background
material) for noise level $\eta=0.005$, $\delta=1\cdot 10^{-6}$ and $\tilde{M}_s=107$.}\label{standard_noise_0_005}
\end{figure}

\begin{figure}[H]
\centering 
\includegraphics[width=0.38\textwidth]{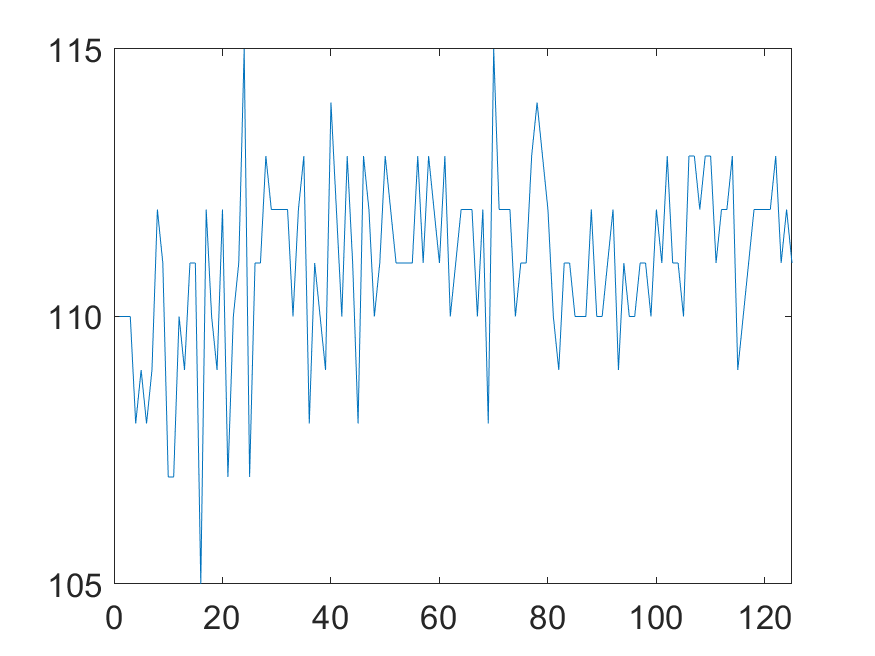}
\includegraphics[width=0.38\textwidth]{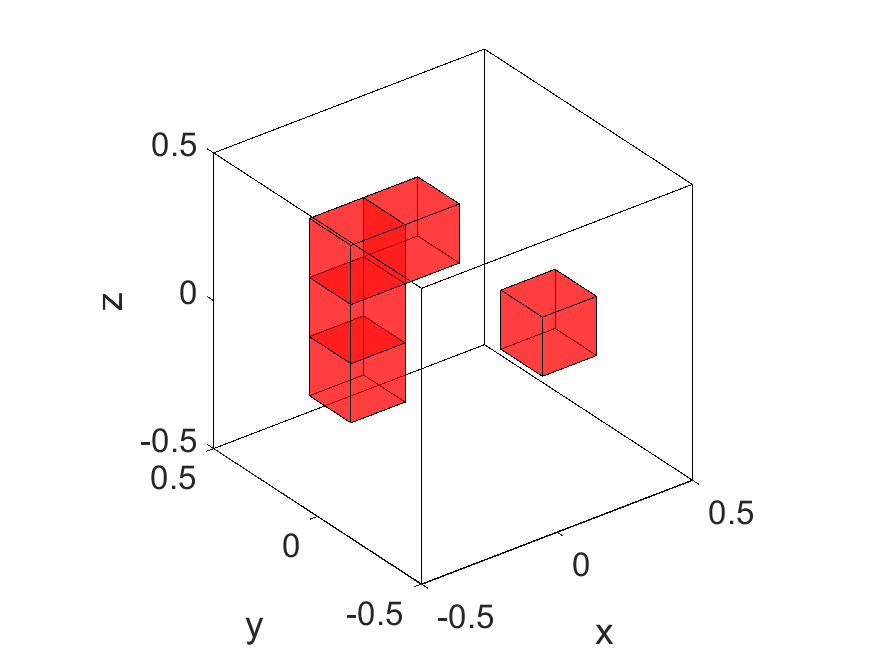}
\caption{Plot of the number of negative eigenvalues (left) and the reconstruction (right) for $\omega=50\,rad/s$ ($l_p=1.79\,m$, $l_s=0.18\,m$ for the homogenous background
material) for noise level $\eta=0.025$, $\delta=4.54\cdot 10^{-6}$ and $\tilde{M}_s=107$.}\label{standard_noise_0_025}
\end{figure}

\noindent
Summarizing the results concerning the standard tests we observe that we can reconstruct the inclusions for the prescribed noise levels $\eta$ with the chosen values of $\delta$ as given in the captions of Figure \ref{standard_noise_0} - Figure \ref{standard_noise_0_025}. 
\\
\\
Finally, we examine the relation of $\tilde{M}_s$, $\delta$ and the noise level $\eta$. As can be seen in the eigenvalue plots of the previous figures (Figure \ref{standard_noise_0} - Figure \ref{standard_noise_0_025}), there exists a distinct gap between the eigenvalues of blocks inside and outside of the inclusion for a noise level $\eta=\delta=0$. It is advantageous for the reconstruction to choose a large $\tilde{M}_s$ for the number of negative eigenvalues which still guarantees a correct reconstruction in the noiseless case. This is denoted as $\tilde{M}_s=M_{\max}$ in Figure \ref{M_delta}. Still, a correct reconstruction at noise level $\eta$ and the corresponding $\delta(\eta)$ is possible for all $M_{\min}\leq \tilde{M}_s\leq M_{\max}$. With increased noise, this gap closes fast, while the $\delta$ to be selected grows almost linearly in our test example up till a noise level of $\eta=3\%$. Here, a reconstruction is still possible, since $M_{\min}(\eta)=M_{\max}(\eta)$. For higher noise levels, a correct reconstruction is no longer possible. Hence, a fixed $\tilde{M}_s=107$ independent of $\eta$ leads to a correct reconstruction for all noise levels satisfying the conditions of Theorem \ref{theo_standard_noisy} or \ref{thm_Lin_inclusionDetection_noisy}.

\begin{figure}[H]
\centering 
\includegraphics[width=0.93\textwidth]{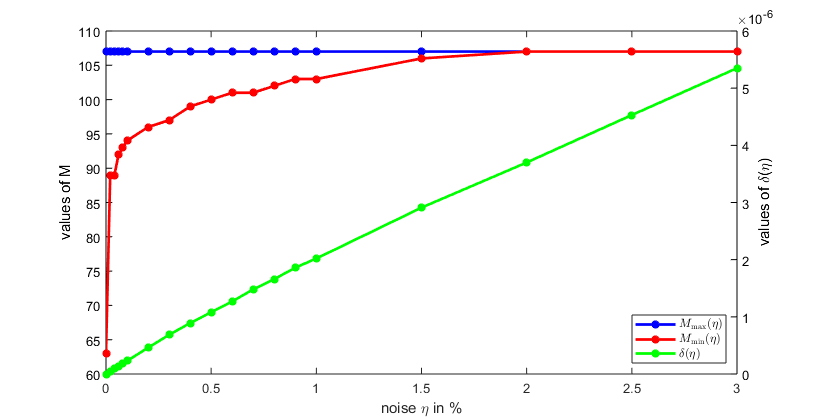}
\caption{Plot of the relation of the bounds $M_{\min}$ and $M_{\max}$ and $\delta$ for different noise levels $\eta$.}\label{M_delta}
\end{figure}

\subsection{Linearized monotonicity test}

Next, we take a look at the reconstructions based on the linearized monotonicity tests and consider the noise levels $\eta=0$ (see Figure \ref{linearized_noise_0}) and $\eta=0.004$ (see Figure \ref{linearized_noise_0_004}) for the frequency $\omega=50\,rad/s$.

\begin{figure}[H]
\centering 
\includegraphics[width=0.40\textwidth]{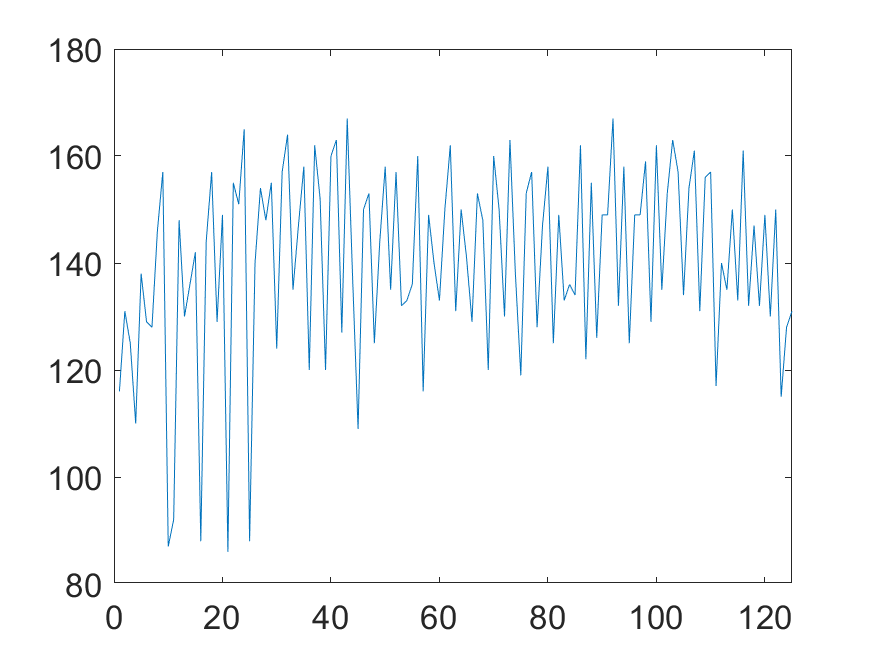}
\includegraphics[width=0.31\textwidth]{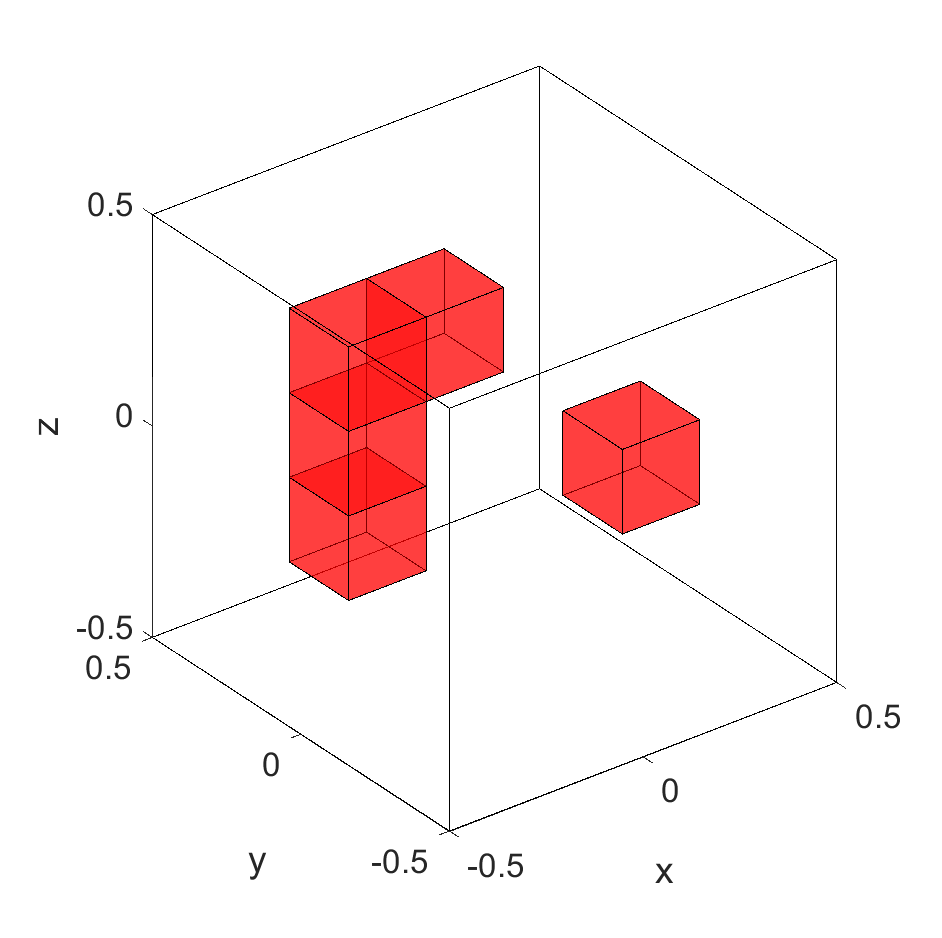}
\caption{Plot of the number of negative eigenvalues (left) and the reconstruction (right) for $\omega=50\,rad/s$ ($l_p=1.79\,m$, $l_s=0.18\,m$ for the homogenous background
material) for noise level $\eta=0$ and $\tilde{M}_l=108$.}\label{linearized_noise_0}
\end{figure}

\begin{figure}[H]
\centering 
\includegraphics[width=0.40\textwidth]{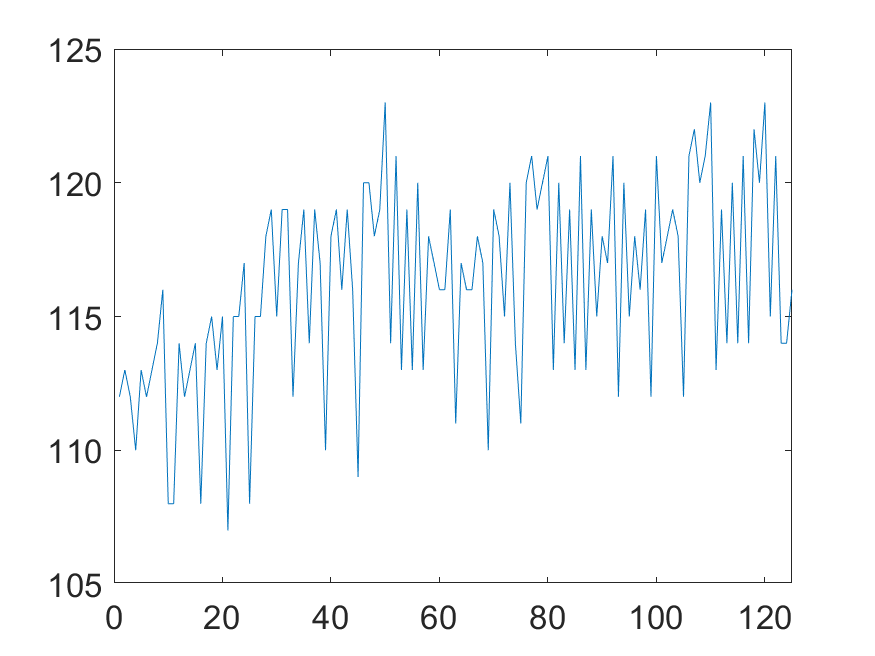}
\includegraphics[width=0.31\textwidth]{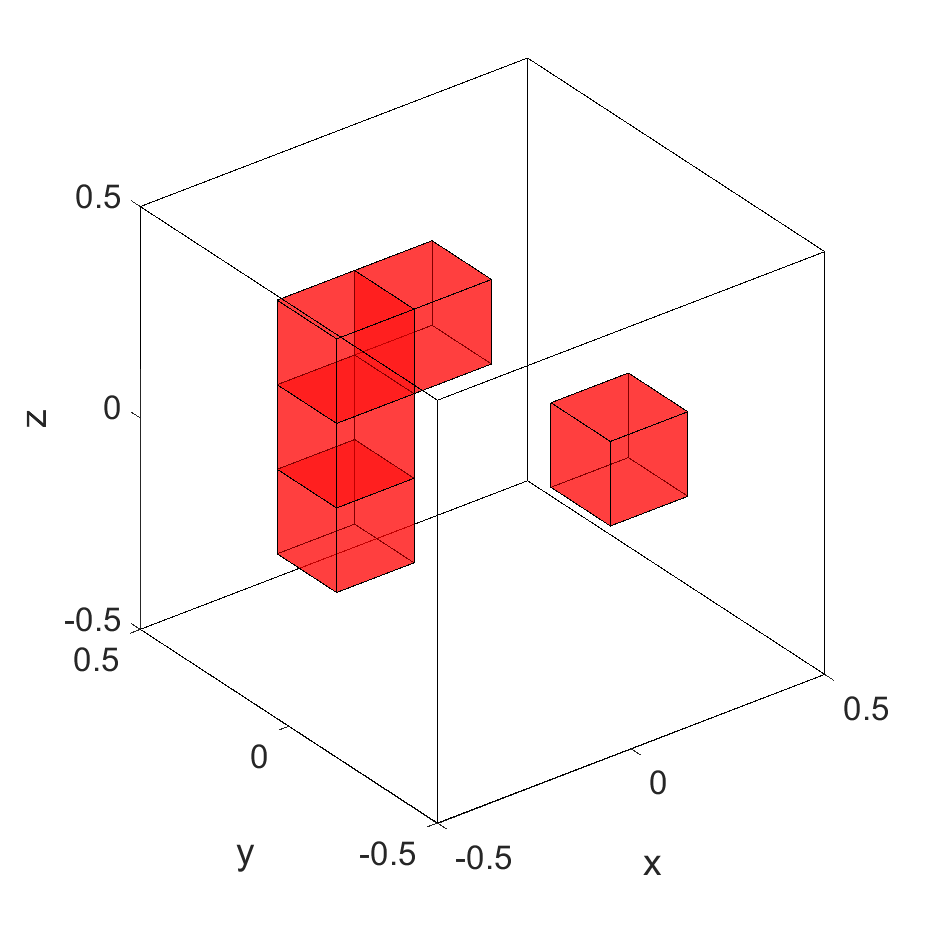}
\caption{Plot of the number of negative eigenvalues (left) and the reconstruction (right) for $\omega=50\,rad/s$ ($l_p=1.79\,m$, $l_s=0.18\,m$ for the homogenous background
material) for noise level $\eta=0.004$, $\delta=1.92\cdot 10^{-6}$ and $\tilde{M}_l=108$.}\label{linearized_noise_0_004}
\end{figure}
\noindent
Comparing the results of the standard monotonicity methods as shown in Figure \ref{standard_noise_0_005} ($\eta=0.005$) as well as Figure \ref{standard_noise_0_025} ($\eta=0.025$) with the one of the linearized monotonicity method in Figure \ref{linearized_noise_0_004} ($\eta=0.004$), we see that the standard tests are more robust w.r.t. noise. However, a reconstruction is still possible. The main advantage of the linearized tests is the extremely reduced computation time. 
\\
\\
Thus, we present a further test model, where we increase the number of boundary loads  to include tangential components ($1500$ boundary loads) and also increase the resolution by using $1000$ pixels compared to the $125$ in the previous figures. The Dirichlet boundary and the Neumann patches remain as before. A comparable calculation for the standard monotonicity tests is not feasible due to the computation time. 
\\
\\
It should be noted, that a correct reconstruction of the inclusion is not possible even in the noiseless case as can be seen in Figure \ref{linearized_noise_0_003_more_boundary_loads}. This is not surprising since those results were already observed in the stationary and oscillatory case in the paper \cite{EH21} and \cite{EP24}. However, a reconstruction for a comparable noise level as used in Figure \ref{linearized_noise_0_more_boundary_loads} still separates the calculated inclusions as well as the general shape and size. It is expected to further increase the resolution by either taking more boundary loads into account (which increases the computation time) or performing the reconstruction based on the simultaneous use of a finite set of distinct frequencies, which will be the subject of further research.

\begin{figure}[H]
\centering 
\includegraphics[width=0.39\textwidth]{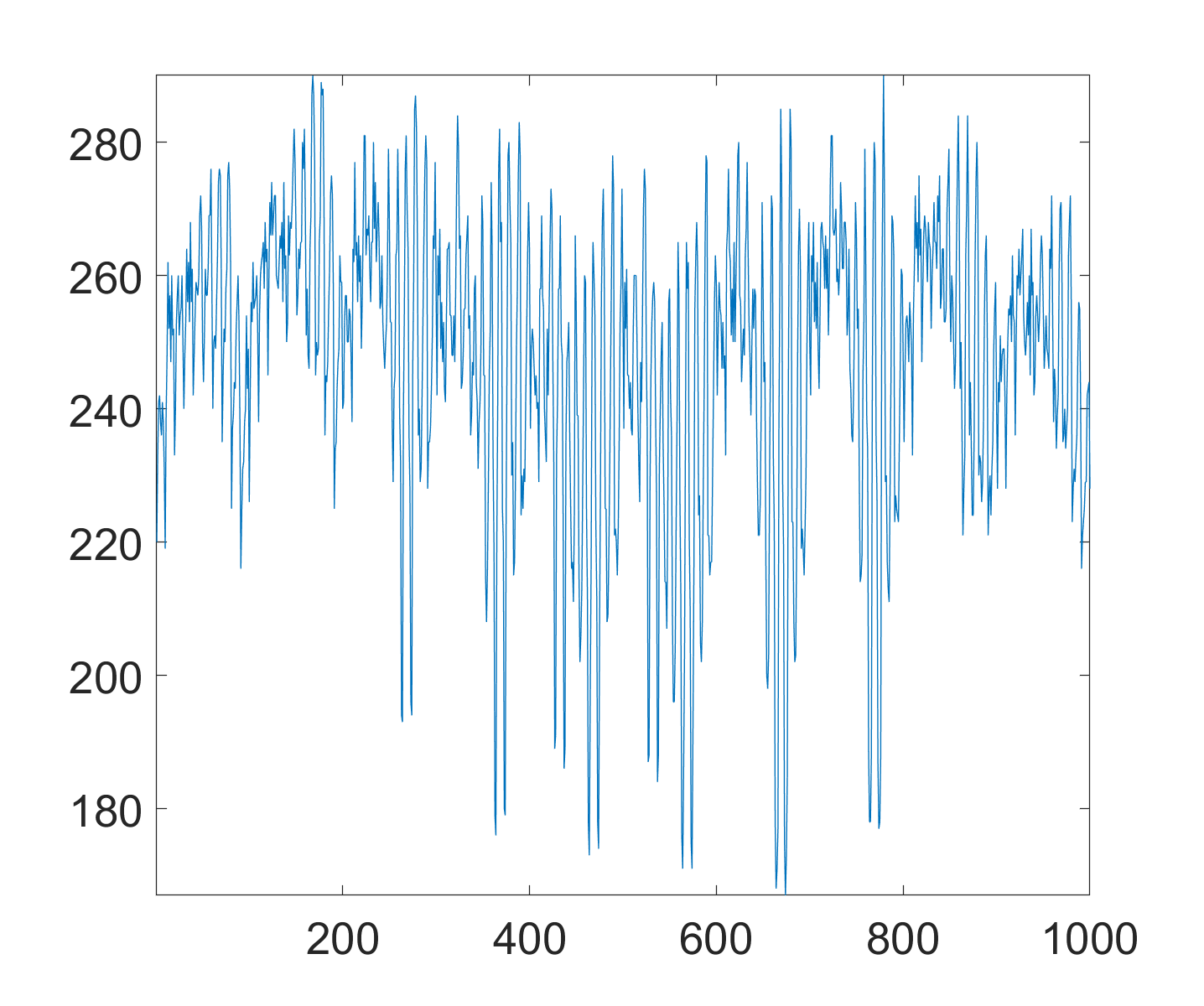}
\includegraphics[width=0.34\textwidth]{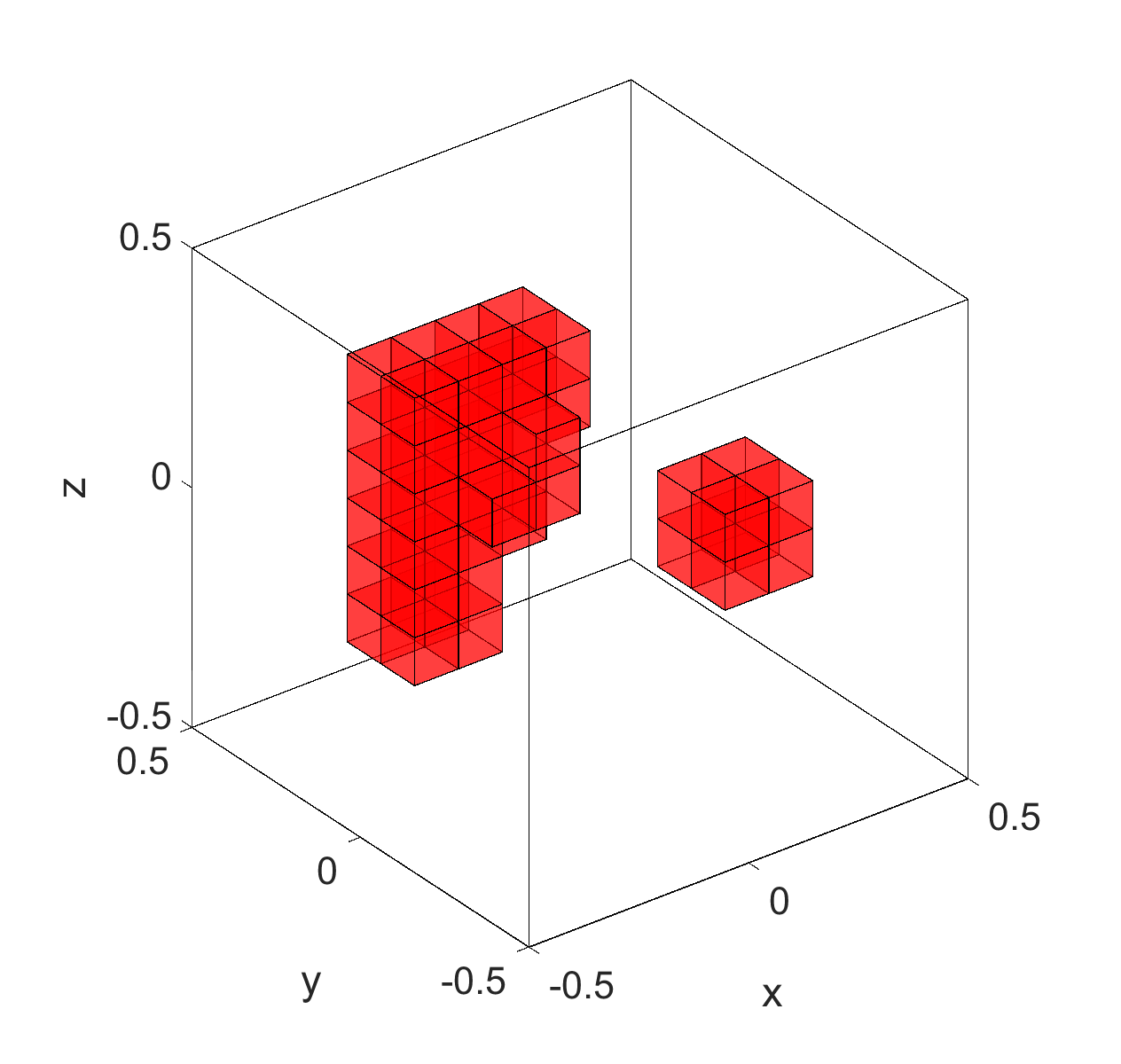}
\caption{Plot of the number of negative eigenvalues (left) and the reconstruction (right) for $\omega=50\,rad/s$ ($l_p=1.79\,m$, $l_s=0.18\,m$ for the homogenous background
material) without noise $\eta=0$ and $\tilde{M}_l=198$.}\label{linearized_noise_0_more_boundary_loads}
\end{figure}

\begin{figure}[H]
\centering 
\includegraphics[width=0.39\textwidth]{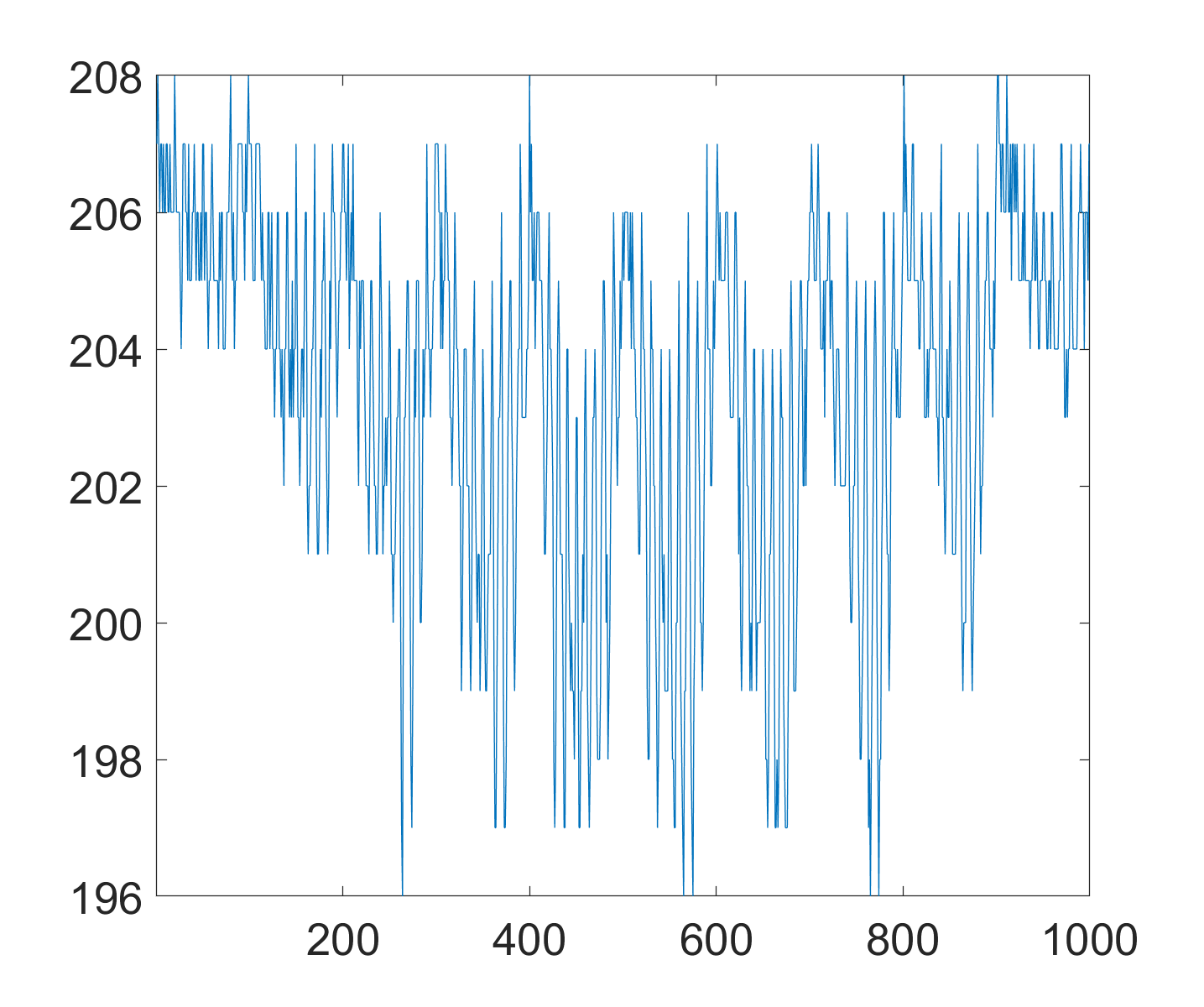}
\includegraphics[width=0.32\textwidth]{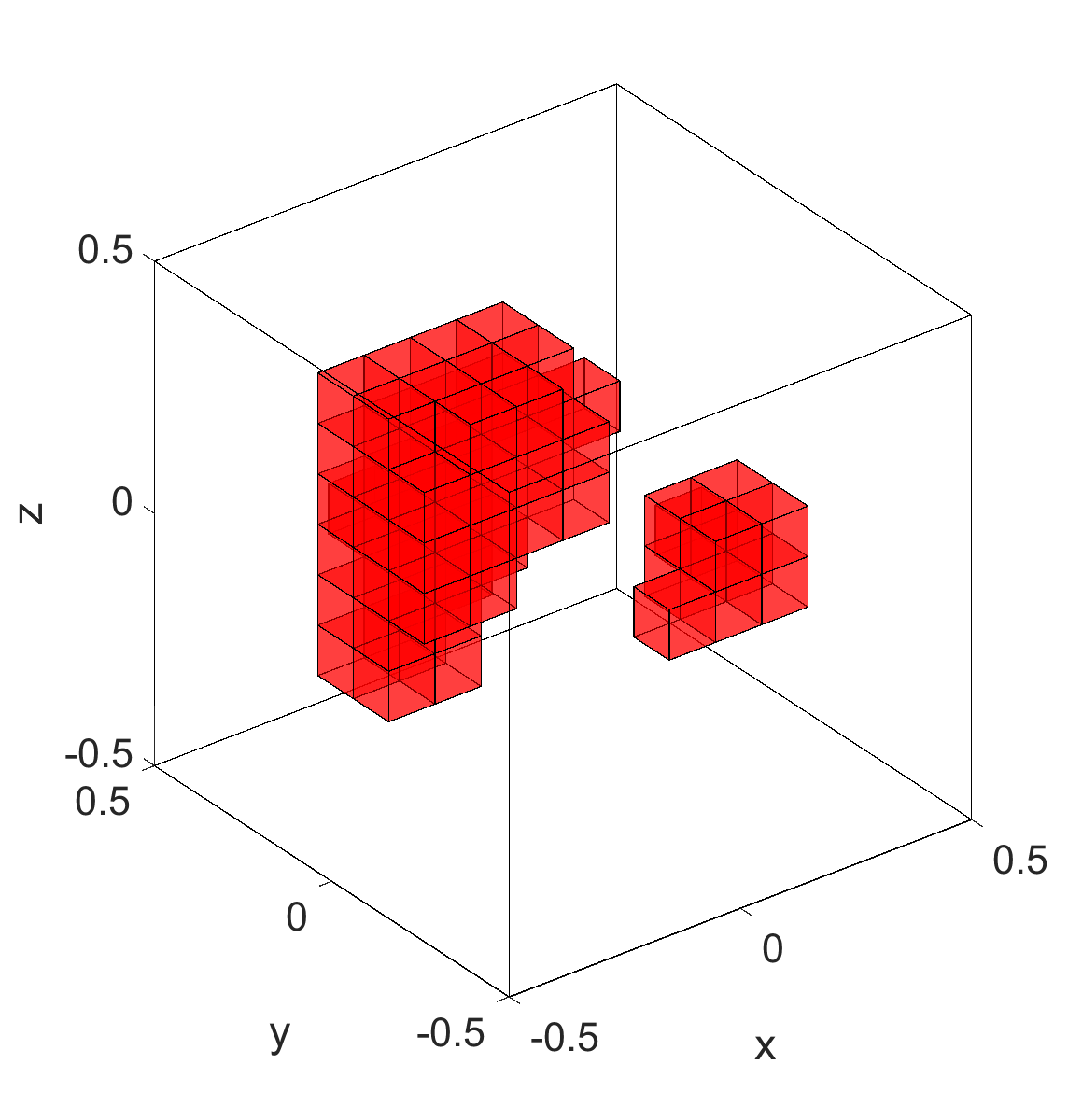}
\caption{Plot of the number of negative eigenvalues (left) and the reconstruction (right) for $\omega=50\,rad/s$ ($l_p=1.79\,m$, $l_s=0.18\,m$ for the homogenous background
material) for noise level $\eta=0.003$, $\delta=6.55\cdot 10^{-6}$ and $\tilde{M}_l=198$.}\label{linearized_noise_0_003_more_boundary_loads}
\end{figure}

\section{Conclusion and Outlook}
\noindent
We showed that the standard as well as linearized monotonicity methods for the time harmonic wave equation can handle noisy data. This is important for simulations based on real world data, e.g., laboratory data. As a future step we will apply our monotonicity methods to laboratory data. Further on, we aim to extend our methods via a monotonicity-based regularization similarly to \cite{EH22}.
\\
\\
\noindent
{\bf{Statements and Declarations}}
\\
The author thanks the German Research Foundation (DFG) for funding the project "Inclusion Reconstruction with Monotonicity-based Methods for the Elasto-oscillatory Wave Equation" (reference number 499303971).

\noindent
\\
\\
\end{document}